\documentclass[reqno]{amsart}
\usepackage{hyperref}
\usepackage{verbatim}

\newcommand{\mysection}[1]{\section{#1}
\setcounter{equation}{0}}

\newtheorem{theorem}{Theorem}[section]
\newtheorem{corollary}[theorem]{Corollary}
\newtheorem{lemma}[theorem]{Lemma}

\theoremstyle{definition}
\newtheorem{remark}[theorem]{Remark}

\theoremstyle{definition}

\theoremstyle{definition}
\newtheorem{assumption}[theorem]{Assumption}

\makeatletter
\def\dashint{\operatorname%
{\,\,\text{\bf--}\kern-.98em\DOTSI\intop\ilimits@\!\!}}
\makeatother

\def\sfg{{\sf g}}

\def\bR{\mathbb{R}}
\def\bZ{\mathbb{Z}}

\def\bM{\mathbb{M}}
\def\bC{\mathbb{C}}

\def\cD{\mathcal{D}}

\def\cH{\mathcal{H}}

\def\cQ{\mathcal{Q}}

\newcommand{\set}[1]{\left\{#1\right\}}

\newcommand{\abs}[1]{\left\lvert#1\right\rvert}

\newcommand{\osc}{\mathop{\hbox{osc}}}

\begin{document}
\title[Parabolic equations]{Parabolic equations with variably partially VMO coefficients}

\author[H. Dong]{Hongjie Dong}
\address[H. Dong]{Division of Applied Mathematics, Brown University, 182 George Street, Providence, RI 02912, USA}
\email{Hongjie\_Dong@brown.edu}
\thanks{The work of the author was partially supported by
NSF Grant DMS-0635607 from IAS and NSF Grant DMS-0800129.}



\begin{abstract}
We prove the $W^{1,2}_{p}$-solvability of second order parabolic equations in nondivergence form in the whole space for $p\in (1,\infty)$. The leading coefficients are assumed to be measurable in
one spatial direction and have vanishing mean oscillation (VMO) in the orthogonal directions and the time variable in each small parabolic cylinder with the direction depending on the cylinder. This extends a recent result by Krylov \cite{Kr08} for elliptic equations and removes the restriction that $p>2$.
\end{abstract}

\maketitle

\mysection{Introduction}
                                        \label{secIntro}
We consider the $W^{1,2}_{p}$-solvability of parabolic equations in nondivergence form:
\begin{equation}
                                                \label{parabolic}
P u-\lambda u=f,
\end{equation}
where $\lambda\ge 0$ is a constant, $f\in L_p$, and $P$ is a uniformly nondegenerate parabolic operator with bounded coefficients:
\begin{equation*}
P u=-u_t+a^{ij}D_{ij}u+b^{j}D_ju+cu.
\end{equation*}

The $L_p$ theory of second order parabolic and elliptic equations has been studied extensively under various regularity assumptions on the coefficients. For equations with uniformly continuous leading coefficients, the solvability has been known for a long time; see, for example, Agmon, Douglis and Nirenberg \cite{ADN64} and Lady\v{z}enskaja, Solonnikov and Ural'ceva \cite{LSU}. With VMO coefficients, the solvability theorems of nondivergence form equations were established in early 1990s by Chiarenza, Frasca and Longo \cite{CFL1,CFL2} and Bramanti and Cerutti \cite{BC93}. The main technical tool in these papers was the theory of singular integrals, in particular, certain estimates of Calder\'on-Zygmund theorem and the Coifman-Rochberg-Weiss commutator theorem. For divergence form equations with VMO/BMO coefficients, we refer the reader to Byun and Wang \cite{ByunWang} and references therein.

On the contrary, the theory of elliptic and parabolic equations with {\em partially} VMO coefficients was quite new and originated by Kim and Krylov \cite{KimKrylov07} and \cite{KK2}.
In \cite{KimKrylov07}, the authors established the $W^2_p$-solvability of elliptic equations in nondivergence form under the assumption that the coefficients $a^{ij}$ are measurable with respect to $x^1$ and VMO with respect to the remaining variables. This result was extended to parabolic equations by the same authors in \cite{KK2}, under the assumption that $a^{ij}$ are measurable with respect to $x^1$ and VMO with respect to the remaining spatial variables and the time variable. The arguments in \cite{KimKrylov07} and \cite{KK2} are based on the method previous developed in Krylov \cite{Krylov_2005}, in which the author gave a unified approach of studying the $L_p$ solvability of both divergence and nondivergence form parabolic equations with leading coefficients measurable in the time variable and VMO
in spatial variables. Unlike the arguments in \cite{CFL1, CFL2, BC93}, the proofs in \cite{Krylov_2005} rely mainly on pointwise estimates of sharp functions of spatial derivatives of solutions (see also \cite{Krylov_2007_mixed_VMO}). We also mention that the results in \cite{KK2, Krylov_2005, Krylov_2007_mixed_VMO} have been improved in Kim \cite{Kim07, Kim07a, Kim07b}, in which most leading coefficients are measurable in the time variable and one spatial variable, and VMO in the other variables. In \cite{DongKrylov}, the $W^2_p$-solvability is obtained for equations with leading coefficients measurable in two spatial variables and VMO in the others, when $p>2$ is sufficiently close to $2$.

The $W^1_p$ and $\cH^1_{p}$ solvability of elliptic and parabolic equations in divergence form with partially BMO coefficients are obtained in recent \cite{DongKim08a} and \cite{Dong08}.

The result in \cite{KimKrylov07} was generalized very recently by Krylov \cite{Kr08} to nondivergence form elliptic equations with {\em variably} partially VMO coefficients. More precisely, the leading coefficients are assumed to be measurable in one direction and VMO in the orthogonal directions in each small ball with the direction depending on the ball. Roughly speaking, the main idea in \cite{Kr08} is to use in a localized way a pointwise sharp function estimate of a portion of the Hessian $D^2u$, proved in \cite{KimKrylov07}, and apply a generalized version of the Fefferman-Stein theorem on sharp functions.

As is pointed out by the author, a restriction of the result in \cite{Kr08} (and also in articles mentioned above regarding nondivergence equations with partially VMO coefficients) is that $p$ has to be greater than $2$. This restriction is due to the following reason. The sharp function estimate in \cite{KimKrylov07} is deduced from the $W^2_2$-solvability of equations with $a^{ij}$ depending  only on $x^1$, which is obtained by using the method of Fourier transforms. In turn, the right-hand side of the estimate contains maximal functions of $q$-th power of $D^2 u$ for some $q>2$, which can be made arbitrarily close to $2$. Therefore, to apply the Fefferman-Stein theorem and the Hardy-Littlewood maximal function theorem one requires $p\ge q>2$.

It is natural to ask is if we still have the $W^2_p$-solvability under the same assumptions of the coefficients when $p$ goes below $2$. In this article, we give a positive answer to this question. We obtain the $W^2_p$-solvability of nondivergence form elliptic equations for any $p\in (1,\infty)$ (stated in Theorem \ref{thm1b}), under the same assumptions as in \cite{Kr08}, i.e. the leading coefficients are variably partially VMO. In fact, we shall establish the corresponding $W^{1,2}_p$-solvability of parabolic equations, from which the result of elliptic equations follows in a standard way. Here the leading coefficient of parabolic equations are assumed to be measurable in
one spatial direction and have vanishing mean oscillation in the orthogonal directions and the time variable in each small parabolic cylinder with the direction depending on the cylinder. For the precise statement of the result, see Theorem \ref{thm1}. 
Furthermore, in the spirit of \cite{Krylov_2007_mixed_VMO}, we also obtain the solvability of parabolic equations in mixed-norm Sobolev spaces $W^{1,2}_{q,p}$ when $q\ge p$ under the same assumptions (see Theorem \ref{thm3}). We note that these results generalize the result in \cite{KK2} as well.

Our proofs follow the approach in \cite{Kr08}. In order to go below $2$, we first establish for any $p\in (1,\infty)$ the $W^{1,2}_p$-solvability of nondivergence form parabolic equations with $a^{ij}$ depending only on $x^1$ (stated in Theorem \ref{thm2}). For this purpose, our idea is that in this situation the equation can be rewritten into a divergence form after a suitable change of variables. This enables us to apply a result recently proved in \cite{Dong08} for divergence form parabolic equations with partially VMO coefficients. Next, to get a sharp function estimate we need to bound the H\"older norms of a portion of $D^2u$ when $u$ satisfies the homogeneous equation (see Theorem \ref{thm3.5}). To this end, we use a bootstrap argument with the aid of Theorem \ref{thm2.2} and an embedding type estimate. We combine Theorem \ref{thm2} and \ref{thm3.5} to prove Theorem \ref{thm1} by applying the aforementioned generalized Fefferman-Stein theorem obtained in \cite{Krylov_2007_mixed_VMO}.

A brief outline of the paper: in the next section, we introduce the notation and state the main results, Theorem \ref{thm1}, \ref{thm1b} and \ref{thm2}. Section \ref{sec3}
contains a few preliminary estimates, including Theorem \ref{thm2.2} and \ref{thm3.5} which are the main ingredients of the proof. We finish the proof of the
$W^{1,2}_p$-solvability in Section \ref{sec4} by combining the results in the previous section. Finally we state and prove the $W^{1,2}_{q,p}$-solvability of parabolic equations in the last section.


\mysection{Notation and main results}
                                        \label{secMain}

We begin the section by introducing some notation. Let $d\ge 2$ be an integer. A typical point in ${\mathbb R}^d$ is
denoted by $x=(x^1,x^2,\cdots,x^d)=(x^1,x')$. We set
$$
D_{i}u=u_{x^i},\quad D_{ij}u=u_{x^ix^j},\quad
D_t u=u_t.
$$
By $Du$ and $D^{2}u$ we
mean the gradient and the Hessian matrix
of $u$. On many occasions we need to take these objects relative to only
part of variables. We also use the following notation:
$$
D_{x'}u=u_{x'},\quad
D_{x^1x'}u=u_{x^1x'},\quad
D_{xx'}u=u_{xx'}.
$$

Throughout the paper, we always assume that $1 < p,q< \infty$
unless explicitly specified otherwise.
By $N(d,p,\cdots)$ we mean that $N$ is a constant depending only
on the prescribed quantities $d, p,\cdots$.
For a (matrix-valued) function $f(t,x)$ in $\bR^{d+1}$, we set
\begin{equation*}
(f)_{\cD} = \frac{1}{|\cD|} \int_{\cD} f(t,x) \, dx \, dt
= \dashint_{\cD} f(t,x) \, dx \, dt,
\end{equation*}
where $\cD$ is an open subset in $\bR^{d+1}$ and $|\cD|$ is the
$d+1$-dimensional Lebesgue measure of $\cD$.
For $-\infty\leq S<T\leq \infty$ and $\Omega\subset\bR^d$, we set
\begin{equation*}
L_{q,p}((S,T)\times \Omega)=L_q((S,T),L_p(\Omega)),
\end{equation*}
i.e., $f(t,x)\in L_{q,p}((S,T)\times \Omega)$ if
\begin{equation*}
\|f\|_{L_{q,p}((S,T)\times \Omega)}=\left(\int_S^T\left(\int_{\Omega}
|f(t,x)|^p \,dx\right)^{q/p}\,dt\right)^{1/q}<\infty.
\end{equation*}
Denote
\begin{align*}
L_p((S,T)\times \Omega)&= L_{p,p}((S,T)\times \Omega),\\
W_{q,p}^{1,2}((S,T)\times \Omega)&=
\set{u:\,u,u_t,Du,D^2u\in L_{q,p}((S,T)\times \Omega)},\\
W_p^{1,2}((S,T)\times \Omega)&=
W_{p,p}^{1,2}((S,T)\times \Omega).
\end{align*}
We also use the abbreviations $L_p=L_p(\bR^{d+1})$, $W^2_p=W^2_p(\bR^{d+1})$, etc. For any $T\in (-\infty,\infty]$ and $\Omega\subset\bR^d$, we denote
\begin{equation*}
\bR_T=(-\infty,T), \quad \bR_T^{d+1}=\bR_T\times \bR^d.
\end{equation*}
Let
$$
B_r(x) = \{ y \in \bR^d: |x-y| < r\}, \quad
Q_r(t,x) = (t-r^2,t) \times B_r(x).
$$
Set $B_r = B_r(0)$, $Q_r=Q_r(0,0)$, and $|B_r|,|Q_r|$ to be the volume of $B_r,Q_r$ respectively.
Let $\cQ=\set{Q_r(t,x): (t,x) \in \bR^{d+1}, r \in (0, \infty)}$.
For a function $g$ defined on $\bR^{d+1}$,
we denote its (parabolic) maximal and sharp function, respectively, by
\begin{align*}
\bM g (t,x) &= \sup_{Q\in \cQ: (t,x) \in Q}
\dashint_{Q} | g(s,y) | \, dy \, ds,\\
g^{\#}(t,x) &= \sup_{Q\in \cQ:(t,x) \in Q}
\dashint_{Q} | g(s,y) - (g)_Q | \, dy \, ds.
\end{align*}


Next we state our assumptions on the coefficients precisely.
We assume that
all the coefficients are bounded and measurable, and $a^{jk}$ are uniformly elliptic, i.e.
$$
|b^j|\le  K,\quad |c|\le K,
$$
\begin{equation}
                        \label{ellipticity}
\delta|\xi|^{2}\le a^{jk}\xi^{j}\xi^{k}\leq\delta^{-1}|\xi|^{2}.
\end{equation}
Denote by $\mathcal A$ the set of $d\times d$ symmetric
matrix-valued measurable functions $\bar a = (\bar a^{ij}(y^1))$ of one spatial variable such that \eqref{ellipticity} holds with $\bar a$ in place of $a$.

Let $\Psi$ be the set of $C^{1,1}$ diffeomorphisms  $\psi: \bR^{d} \to \bR^{d}$ such that the mappings $\psi$  and $\phi=\psi^{-1}$ satisfy
\begin{equation}
                                    \label{eq10.54}
|D\psi| + |D^2\psi| \le\delta^{-1},\quad
|D\phi| + |D^2\phi| \le\delta^{-1}.
\end{equation}

\begin{assumption}[$\gamma$]                          \label{assump2}
There exists a positive constant $R_0\in(0,1]$ such that, for
any parabolic cylinder $Q$ of radius less than $R_0$,
one can find an $\bar a\in \mathcal A$ and a  $\psi = (\psi^1,\cdots,\psi^d)\in\Psi$ such that
\begin{equation}
                            \label{eq3.07}
\int_Q |a(t,x)-\bar a( \psi^1(x))|\, dx \,dt
\le\gamma|Q|.
\end{equation}
\end{assumption}

Next we state the main result of the article.

\begin{theorem}
                            \label{thm1}
For any
$p\in (1,\infty)$ there exists a $\gamma =\gamma(d,\delta,p) >0$
such that under Assumption
\ref{assump2} ($\gamma$) for any $T\in (-\infty,+\infty]$ the following holds.

i) For any $u\in W^{1,2}_p(\bR^{d+1}_T)$,
$$
\lambda\|u\|_{L_{p}(\bR^{d+1}_T)}+\sqrt{\lambda}
\|Du\|_{L_{p}(\bR^{d+1}_T)}+\|D^{2}u\|_{L_{p}(\bR^{d+1}_T)}
+\|u_t\|_{L_{p}(\bR^{d+1}_T)}
$$
\begin{equation*}
\leq N\|Pu-\lambda u\|_{L_{p}(\bR^{d+1}_T)},
\end{equation*}
provided that $\lambda\geq \lambda_0$, where $\lambda_0
\geq0$ and $N$ depend only
 on $d,\delta,p,K$, and $R_0$.

ii) For any $\lambda> \lambda_0$ and $f\in L_p(\bR^{d+1}_T)$, there exists a unique solution
$u\in W^{1,2}_p(\bR^{d+1}_T)$ of equation \eqref{parabolic} in
$\bR^{d+1}_T$.
\end{theorem}

Theorem \ref{thm1} yields the following solvability result of the initial value problem of parabolic equations (see, for instance, \cite{Krylov_2005}).

\begin{theorem}
                            \label{thm1b}
For any
$p\in (1,\infty)$ there exists a $\gamma =\gamma(d,\delta,p) >0$
such that under Assumption
\ref{assump2} ($\gamma$) for any $T\in (0,\infty)$ the following holds. For any $f\in L_p((0,T)\times\bR^{d})$, there exists a unique solution $u\in W^{1,2}_p((0,T)\times \bR^{d})$ of
$$
Pu=f\,\,\,\text{in} \,\,(0,T)\times \bR^{d},\quad u(0,\cdot)=0.
$$
Moreover, we have
$$
\|u\|_{W^{1,2}_p((0,T)\times \bR^{d})}\le N
\|f\|_{L_p((0,T)\times \bR^{d})},
$$
where $N$ depend only on $d,\delta,p,K,T$ and $R_0$.
\end{theorem}

In the case that all the coefficients are time-independent, we also consider the $W^{2}_{p}$-solvability of elliptic equations in nondivergence
form:
\begin{equation}
                                                \label{elliptic}
L u-\lambda u=f,
\end{equation}
where
$$ L u=a^{jk}D_{jk}u+b^{j}D_ju+cu.
$$

\begin{assumption}[$\gamma$]                          \label{assump3}
There exists a positive constant $R_0\in(0,1]$ such that for
any ball $B$ of radius less than $R_0$ one can find an $\bar a\in \mathcal A$ and a  $\psi = (\psi_1,\cdots,\psi_d)\in \Psi$ such that
\begin{equation*}
\int_B |a(x)-\bar a( \psi^1(x))|\, dx
\le\gamma|B|.
\end{equation*}
\end{assumption}

The following $W^2_p$-solvability theorem for elliptic equations is an immediate corollary of Theorem \ref{thm1}, which generalizes Theorem 1.4 of \cite{Kr08} by dropping the condition $p>2$.

\begin{theorem}
                            \label{thm2}
For any
$p\in (1,\infty)$ there exists a $\gamma =\gamma(d,\delta,p) >0$
such that under Assumption
\ref{assump3} ($\gamma$) the following holds.

i) For any $u\in W^{2}_p(\bR^{d})$,
\begin{equation*}
\lambda\|u\|_{L_{p}(\bR^{d})}+\sqrt{\lambda}
\|Du\|_{L_{p}(\bR^{d})}+\|D^{2}u\|_{L_{p}(\bR^{d})}
\leq N\|Lu-\lambda u\|_{L_{p}(\bR^{d})},
\end{equation*}
provided that $\lambda\geq \lambda_0$, where $\lambda_0
\geq0$ and $N$ depend only
 on $d,\delta,p,K$, and $R_0$.

ii) For any $\lambda> \lambda_0$ and $f\in L_p(\bR^{d})$, there exists a unique solution
$u\in W^{2}_p(\bR^{d})$ of equation \eqref{elliptic} in
$\bR^{d}$.
\end{theorem}
\begin{proof}
The theorem follows from Theorem \ref{thm1} using the idea that solutions to elliptic equations can be viewed as steady state solutions to parabolic equations. We omit the details and refer the reader to the proof of Theorem 2.6 \cite{Krylov_2005}.
\end{proof}

The theorems above generalizes several previously known results of nondivergence form equations with discontinuous coefficients to a large extent. In particular, one can get the solvability of equations to which the results in \cite{KimKrylov07} and \cite{KK2} are not applicable. We refer the reader to an interesting example given in the end of the Introduction of \cite{Kr08}. Here we give another example when $d=2$.

Consider in the polar coordinates $(\rho,\theta)$ the union of two graphs $\rho=e^{-\theta/\epsilon}$ and $\rho=e^{-(\theta+\pi)/\epsilon}$ for some $\epsilon>0$.
This curve divides the plane into two connected components. It becomes
flat if $\epsilon$ is small. Let $a^{ij}$ be different constants
in these two components.
If $\epsilon=\epsilon(\gamma)$ is sufficiently small, it is easy to check that $a^{ij}$
satisfy Assumption \ref{assump3} ($\gamma$). But there does not exist a common diffeomorphism $\psi$ which works for all small balls centered at the origin. Therefore, the result in \cite{KimKrylov07}  is not applicable in this case even if one uses a partition of the unity.

\begin{remark}
In \cite{DongKrylov}, the $W^2_p$-solvability is obtained for equations with leading coefficients measurable in two spatial variables and VMO in the others, when $p>2$ is sufficiently close to $2$. An interesting problem is whether that result can be extended to equations with variably partially VMO coefficients.
\end{remark}

\mysection{Preliminaries}
					\label{sec3}

First we recall the following embedding-type result (see, for example, \cite{Krylov_2007_mixed_VMO} or \cite{LSU}).
\begin{lemma}
                                                    \label{lem3.4}
Let $q\ge 1$ and
\begin{equation*}
\frac 1 q<\frac 1 {d+2}+\frac 1 p.
\end{equation*}
Then there is a constant $N=N(d,p,q,r,R)$ such that for any $u\in W_{q,\text{loc}}^{1,2}$ and $0<r<R<\infty$ we have
\begin{equation*}
\|u\|_{L_p(Q_r)}+\|Du\|_{L_p(Q_r)}\le N\|u\|_{W^{1,2}_q(Q_{R})}.
\end{equation*}
\end{lemma}

Let
$$
P_0u=-u_t+a^{ij}D_{ij}u,
$$
where $a^{ij}=a^{ij}(x^1)$. Our proof relies on the following solvability theorem.
\begin{theorem}
						\label{thm2.2}
Let $p\in (1,\infty)$ and $T\in (-\infty,\infty]$. Then for any $u\in W^{1,2}_p(\bR^{d+1}_T)$ and $\lambda\ge 0$, we have
$$
\lambda\|u\|_{L_{p}(\bR^{d+1}_T)}+\sqrt{\lambda}
\|Du\|_{L_{p}(\bR^{d+1}_T)}+\|D^2u\|_{L_{p}(\bR^{d+1}_T)}
+\|u_{t}\|_{L_{p}(\bR^{d+1}_T)}
$$
\begin{equation}
                \label{eq10.30}
\le N\|P_0u-\lambda u\|_{L_{p}(\bR^{d+1}_T)},
\end{equation}
where $N=N(d,p,\delta)>0$.
Moreover, for any $f\in L_p(\bR^{d+1}_T)$ and $\lambda>0$ there is a unique $u\in W^{1,2}_p(\bR^{d+1}_T)$ solving $P_0 u-\lambda  u=f$ in $\bR^{d+1}_T$.
\end{theorem}

\begin{proof}
First we assume $T=\infty$. By the method of continuity, it suffices to prove the a priori estimate \eqref{eq10.30} for $u\in C_0^\infty$. Let
\begin{equation}
 				\label{eq2.32}
f=P_0 u-\lambda u.
\end{equation}
The idea is to use the solvability of the corresponding divergence form operator. We make a change of variables:
$$
y^1=\varphi(x^1):=\int_0^{x^1} \frac 1 {a^{11}(s)}\,ds,\quad y^j=x^j,\,\,j\ge 2.
$$
It is easy to see that $\varphi$ is a bi-Lipschitz function and
\begin{equation*}
\delta \le y^1/x^1\le \delta^{-1},\quad D_{y^1}=a^{11}(x^1)D_{x^1}.
\end{equation*}
Denote
$$
v(t,y^1,y')=u(t,\varphi^{-1}(y^1),y'),\quad
\tilde a^{ij}(y^1)=a^{ij}(\varphi^{-1}(y^1)),
$$
$$
\tilde f(t,y)=f(t,\varphi^{-1}(y^1),y').
$$
Define a divergence form operator $\tilde P_0$ by
$$
\tilde P_0 v=-v_t+D_{1}\left(\frac 1 {\tilde a^{11}} D_1 v\right)+\sum_{j=2}^d D_{j}\left(\frac {\tilde a^{1j}+\tilde a^{j1}}
{\tilde a^{11}} D_1v
\right)+\sum_{i,j=2}^d D_{j}(\tilde a^{ij}D_iv).
$$
Clearly, $v$ satisfies in $\bR^{d+1}$
$$
\tilde P_0 v-\lambda v=\tilde f.
$$
By Corollary 5.5 of \cite{Dong08}, we have
\begin{equation*}
\lambda \| v \|_{L_p}
+ \sqrt\lambda\|Dv\|_{L_p}
\le N\|\tilde f\|_{L_p}.
\end{equation*}
Therefore,
\begin{equation}
					\label{eq2.25}
\lambda \|u \|_{L_p}
+ \sqrt\lambda\|Du\|_{L_p}
\le N\|f\|_{L_p}.
\end{equation}
Next we estimate $D^2 u$. Notice that for each $k=2,...,d$ $D_kv$ satisfies
$$
\tilde P_0 (D_k v)-\lambda D_k v=D_k \tilde f.
$$
Again by using Corollary 5.5 of \cite{Dong08}, we get
$$
\| D_{yy^k} v \|_{L_p}
\le N\|\tilde f\|_{L_p},
$$
which implies
\begin{equation}
				\label{eq2.57}
\| D_{xx'} u \|_{L_p}
\le N\|f\|_{L_p}.
\end{equation}
Finally, to estimate $D_1^2 u$, we return to the equation in the original coordinates. From \eqref{eq2.32}, we see that $w:=D_1 u$ satisfies
\begin{equation*}
-w_t+D_1(a^{11}D_1w)+\Delta_{d-1}w-\lambda w=D_1 f+\sum_{ij>1}D_1
\left((\delta_{ij}-a^{ij})D_{ij}u\right).
\end{equation*}
We use Corollary 5.5 of \cite{Dong08} again to get
\begin{equation}
				\label{eq2.36}
\| D_{1}^2 u \|_{L_p}\le \| Dw \|_{L_p}\le N\| f \|_{L_p}
+N\sum_{ij>1} \| D_{ij}u \|_{L_p}.
\end{equation}
Combining \eqref{eq2.25}, \eqref{eq2.57} and \eqref{eq2.36} yields \eqref{eq10.30} by bearing in mind that
\begin{equation*}
u_t=a^{ij}D_{ij}u-\lambda u-f.
 \end{equation*}

For general $T\in (-\infty,\infty]$, we use the fact that $u=w$ for $t<T$, where $w\in
W_p^{1,2}$ solves
$$
P_0 w-\lambda w=\chi_{t<T}(P_0 u-\lambda u).
$$ The theorem is proved.
\end{proof}

As a direct consequence of Theorem \ref{thm2.2}, we have:

\begin{corollary}
						\label{cor2.3}
Let $p\in (1,\infty)$ and $T\in (0,\infty)$. Then for any $f\in L_p((0,T)\times\bR^{d})$ and $\lambda>0$ there is a unique $u\in W^{1,2}_p((0,T)\times\bR^{d})$ solving $P_0 u-\lambda  u=f$ in $(0,T)\times\bR^{d}$ and $u(0,\cdot)=0$. Moreover, we have
$$
\lambda\|u\|_{L_{p}((0,T)\times\bR^{d})}+\sqrt{\lambda}
\|Du\|_{L_{p}((0,T)\times\bR^{d})}+\|D^2u\|_{L_{p}
((0,T)\times\bR^{d})}
+\|u_{t}\|_{L_{p}((0,T)\times\bR^{d})}
$$
\begin{equation*}
\le N\|P_0u-\lambda u\|_{L_{p}((0,T)\times\bR^{d})},
\end{equation*}
where $N=N(d,p,\delta,T)>0$.
\end{corollary}

\begin{corollary}
                                    \label{cor3.3}
Let $p\in (1,\infty)$  and
$u \in W_{p,\text{loc}}^{1,2}$.
Then for any $(t_0,x_0)\in \bR^{d+1}$ and $0 < r<R < \infty$,
\begin{multline*}
\| u_t \|_{L_p(Q_r(t_0,x_0))}
+ \|D^2u\|_{L_p(Q_r(t_0,x_0))}\\
\le N \left(\| P_0 u \|_{L_p(Q_{R}(t_0,x_0))}
+ \|u \|_{L_p(Q_{R}(t_0,x_0))} + \|Du\|_{L_p(Q_{R}(t_0,x_0))}\right),
\end{multline*}
where $N = N(d,\delta,r,R)$.
\end{corollary}
\begin{proof}
It suffices to apply Theorem \ref{thm2.2} on $u\eta$, where $\eta$ is a suitable cutoff function.
\end{proof}

We now state and prove the following useful theorem, which gives a H\"older estimate of $D_{xx'}u$.
\begin{theorem}
                                            \label{thm3.5}
Let $p\in (1,\infty)$ and $u\in C_0^\infty$. Assume $P_0u=0$ in $Q_2(t_0,x_0)$ for some $(t_0,x_0)\in \bR^{d+1}$. Then for any $\alpha\in (0,1)$, we have
\begin{equation}
                                            \label{eq27.4.18}
\|D_{xx'}u\|_{C^{\alpha/2,\alpha}(Q_1(t_0,x_0))}\le N(\|u\|_{L_p(Q_2(t_0,x_0))}+\|Du\|_{L_p(Q_2(t_0,x_0))})
\end{equation}
for some constant $N=N(d,\delta,\alpha,p)$.
\end{theorem}
\begin{proof}
Without loss of generality, we may assume $(t_0,x_0)=(0,0)$. We will prove the lemma by a bootstrap argument. Take an increasing sequence $p_j\in (1,\infty)$, $j = 0, 1, \cdots, m$, where $m$ depends only on $d$ and $\alpha$, such that
$$p_0 = p,\quad p_m>\frac {d+2}{1-\alpha},\quad \frac 1 {p_{j+1}}<\frac 1 {p_j}-\frac 1 {d + 2}.$$
Also we take a sequence of shrinking cylinders $$
Q^{(j)},\quad j=0,1,2,\cdots,3m+2$$ such that $Q^{(0)}=Q_2$ and $Q^{(3m+2)}=Q_1$.

By Corollary \ref{cor3.3}, we have
\begin{equation}
                                    \label{eq10.53}
\|u\|_{W^{1,2}_{p_0}(Q^{(1)})}\le N(\|u\|_{L_p(Q^{(0)})}+\|Du\|_{L_p(Q^{(0)})}):=NI.
\end{equation}
Then it follows from Lemma \ref{lem3.4} that
\begin{equation}
                                    \label{eq27.4.46}
\|u\|_{L_{p_1}(Q^{(2)})}+
\|Du\|_{L_{p_1}(Q^{(2)})}\le
N\|u\|_{W^{1,2}_{p_0}(Q^{(1)})}
\le NI.
\end{equation}
Since $D_{x'}u$ satisfies the same equation in $Q_2$, we have by \eqref{eq27.4.46} and \eqref{eq10.53}
\begin{equation}
                                    \label{eq27.4.49}
\|D_{xx'}u\|_{L_{p_1}(Q^{(3)})}
\le N(\|D_{x'}u\|_{L_{p_0}(Q^{(1)})}+\|D_{xx'}u\|_{L_{p_0}(Q^{(1)})})
\le NI.
\end{equation}
Now we write
$$
-u_t+a^{11}D_{11}u+\Delta_{d-1}u=\Delta_{d-1}u
-\sum_{ij>1}a^{ij}D_{ij}u
$$ and use Corollary \ref{cor3.3} with $p_1$ in place of $p$. This together with \eqref{eq27.4.46} and \eqref{eq27.4.49} gives
\begin{equation}
                                    \label{eq27.5.02}
\|u\|_{W^{1,2}_{p_1}(Q^{(4)})}
\le N(\|u\|_{L_{p_1}(Q^{(3)})}+\|Du\|_{L_{p_1}(Q^{(3)})}
+\|D_{xx'}u\|_{L_2(Q^{(3)})})
\le NI.
\end{equation}
We can iterate \eqref{eq27.4.46}-\eqref{eq27.5.02} along with the increasing sequence $p_j$ and shrinking cylinders $Q^{(j)}$. After $m$ steps, we reach
\begin{equation*}
\|u\|_{W_{p_m}^{1,2}(Q^{(3m+1)})}\le NI.
\end{equation*}
Again, since $D_{x'}u$ satisfies the same equation in $Q_2$, we have
\begin{equation}
                                    \label{eq11.32b}
\|D_{x'}u\|_{W_{p_m}^{1,2}(Q^{(3m+2)})}\le N(\|D_{x'}u\|_{L_{p}(Q^{(1)})}+\|D_{xx'}u\|_{L_{p}(Q^{(1)})})
\le NI.
\end{equation}
Finally, due to the classical Sobolev embedding theorem of parabolic type, \eqref{eq11.32b} gives
\begin{equation*}
\|D_{xx'}u\|_{C^{\alpha/2,\alpha}(Q_1)}\le NI.
\end{equation*}
The theorem is proved.
\end{proof}

\begin{corollary}
                                        \label{cor3.6}
Let $p\in (1,\infty)$ and $u\in C_0^\infty$. Suppose $ P_0 u=0$ in $Q_2(t_0,x_0)$ for some $(t_0,x_0)\in \bR^{d+1}$. Then for any $\alpha\in (0,1)$, we have
\begin{equation}
                                            \label{eq27.5.24}
[D_{xx'}u]_{C^{\alpha/2,\alpha}(Q_1(t_0,x_0))}\le N(\|u_t\|_{L_p(Q_2(t_0,x_0))}+\|D^2 u\|_{L_p(Q_2(t_0,x_0))}).
\end{equation}
for some constant $N=N(d,\delta,\alpha,p)$.
\end{corollary}
\begin{proof}
Again we assume $(t_0,x_0)=(0,0)$.
We notice that $v:=u-(u)_{Q_2}-x^j(D_ju)_{Q_2}$ satisfies the same equation as $u$ in $Q_2$.
Therefore, by \eqref{eq27.4.18},
$$
[D_{xx'}u]_{C^{\alpha/2,\alpha}(Q_1)}=[D_{xx'}v]_{C^{\alpha/2,\alpha}(Q_1)}
$$
\begin{equation}
                                            \label{eq10.10}
\le N\|u-(u)_{Q_2}-x^j(D_ju)_{Q_2}\|_{L_p(Q_2)}
+\|Du-(Du)_{Q_2}\|_{L_p(Q_2)}.
\end{equation}
By Lemma 5.4 of \cite{Krylov_2007_mixed_VMO}, we see that the right-hand side of \eqref{eq10.10} is less than the right-hand side of \eqref{eq27.5.24}. The corollary is proved.
\end{proof}

\begin{corollary}
                                        \label{cor3.7}
Let $p\in (1,\infty)$, $\kappa\ge 2$, $r\in (0,\infty)$ and $u\in C_0^\infty$. Assume $ P_0 u-\lambda u=0$ in $Q_{\kappa r}(t_0,x_0)$ for some $(t_0,x_0)\in \bR^{d+1}$. Then for any $\alpha\in (0,1)$, we have
\begin{equation*}
\left(|D_{xx'}u-(D_{xx'}u)_{Q_r(t_0,x_0)}|^p\right)_{Q_r(t_0,x_0)}
\le N\kappa^{-\alpha p}
\left(|u_t|^p+|D^2u|^p\right)_{Q_{\kappa r}}.
\end{equation*}
for some constant $N=N(d,\delta,\alpha,p)$.
\end{corollary}
\begin{proof}
By a scaling argument, it suffices to consider the case $r=2/\kappa\le 1$. Due to Corollary \ref{cor3.6}, we get
$$
\left(|D_{xx'}u-(D_{xx'}u)_{Q_r(t_0,x_0)}|^p
\right)_{Q_{2/\kappa}(t_0,x_0)}
\le N\kappa^{-p\alpha} [D_{xx'}u]_{C^{\alpha/2,\alpha}(Q_{2/\kappa}(t_0,x_0))}
$$
$$
\le N\kappa^{-p\alpha}
[D_{xx'}u]_{C^{\alpha/2,\alpha}(Q_1(t_0,x_0))}
\le N\kappa^{-p\alpha}\left(|u_t|^p+|D^2u|^p\right)_{Q_{2}}.
$$
\end{proof}

\begin{lemma}
                                     \label{lemma 9.10.1}
Let $p\in (1,\infty)$, $\kappa \ge 2$, and $r > 0$.
Assume that $u \in
C_0^{\infty}$  and $P_0 u =0$ in $Q_{\kappa r}$.
Then there exist constants
$N = N(d,p, \delta)$ and $\alpha=\alpha(d,p,\delta)\in (0,1]$ such that
\begin{equation*}
\left(|u_t  -
 (u_t)_{Q_r}|^p \right)_{Q_r}
\le N \kappa^{-p\alpha} \left(|u_t|^{p}\right)_{Q_{\kappa r}}.
\end{equation*}
\end{lemma}

\begin{proof}
By using scaling
we reduce the general situation to the one in which $r=1$. Since $Lu_t=0$ in $Q_{\kappa r}$, by Lemma 4.2.4 of \cite{Kr85} and Theorem 7.21 of \cite{Li}
$$
\osc _{Q_{1/\kappa}}u_t
\le N\kappa^{-\alpha} \|u_t \|_{L_p(Q_1)}
$$
with $\alpha$ and $N$ as in the statement. Scaling this estimate
shows that
$$
\osc _{Q_{1}}u_t
\le N \kappa^{-\alpha}\left(| u_t|^{p}\right)_{Q_{\kappa }}^{1/p}.
$$
It only remains to observe that
$$
\left(| u_t  -
 ( u_t)_{Q_1}|^p \right)_{Q_1}\leq N
 (\osc _{Q_{1}}u_t)^p.
$$
The lemma is proved.
\end{proof}

\begin{theorem}
                                        \label{thm3.7}
Let $p\in (1,\infty)$, $\kappa \ge 4$, $r > 0$ and $u\in C_0^\infty$. Let $\alpha$ be the constant in Lemma \ref{lemma 9.10.1}. Then for any $(t_0,x_0)\in \bR^{d+1}$, we have
$$
\left(|D_{xx'}u-(D_{xx'}u)_{Q_r(t_0,x_0)}|^p\right)_{Q_r(t_0,x_0)}
+\left(|u_t-(u_t)_{Q_r(t_0,x_0)}|^p\right)_{Q_r(t_0,x_0)}
$$
\begin{equation*}
\le N\kappa^{d+2}\left(|P_0u|^p\right)_{Q_{\kappa r}(t_0,x_0)}+N\kappa^{-\alpha p}
\left(|u_t|^p+|D^2 u|^p\right)_{Q_{\kappa r}(t_0,x_0)}.
\end{equation*}
for some constant $N=N(d,\delta,p)$,
\end{theorem}
\begin{proof}
The theorem follows from Corollary \ref{cor2.3}, Corollary \ref{cor3.7} and Lemma \ref{lemma 9.10.1}; see, for instance, the proof of Theorem 4.5 \cite{DongKrylov}.
\end{proof}

We finish this section by recalling a generalized version of the Fefferman-Stein theorem proved in \cite{Kr08}.
To state this theorem,
let
$$
\bC_n=\{C_n(i_0,i_1,\cdots,i_d),i_0,\cdots, i_d\in \bZ\},\quad n\in \bZ
$$
be the filtration of partitions given by parabolic dyadic cubes, where
\begin{multline*}
C_n(i_0,i_1,\cdots, i_d)\\
=[i_0 2^{-2n},(i_0+1) 2^{-2n})\times [i_1 2^{-n}, (i_1 + 1)2^{-n})\times \cdots \times [i_d 2^{-n}, (i_d + 1)2^{-n}).
\end{multline*}
\begin{theorem}
                                    \label{generalStein}
Let $p\in (0,1)$, $U,V,H\in L_1$. Assume $V\ge |U|$, $H\ge 0$ and for any $n\in \bZ$ and $C\in \bC_n$ there exists a measurable function $U^C$ given on $C$ such that $|U|\le U^C\le V$ on $C$ and
$$
\min\left\{\int_C|U-(U)_C|\,dx\,dt,
\int_C|U^C-(U^C)_C|\,dx\,dt\right\}
\le \int_C H\,dx\,dt.
$$ Then we have
\begin{equation*}
\|U\|_{L_p}^p\leq N\|H\|_{L_p}\|V\|_{L_p}^{p-1},
\end{equation*}
provided that $H,V\in L_p$.
\end{theorem}

\mysection{Proof of Theorem \ref{thm1}}
                        \label{sec4}

With the preparations in the previous section, we complete the proof of Theorem \ref{thm1} by using the idea in \cite{Kr08}.

Following the proof of Lemma 3.2 \cite{Kr08} with obvious modifications, we can deduce the following lemma from Theorem \ref{thm3.7}.

\begin{lemma}
                                    \label{lem4.1}
Let $\bar a\in \mathcal A$ and $\psi\in \Psi$. Let
$\alpha$ be the constant in
Lemma \ref{lemma 9.10.1}. Denote
$$
\hat Pu(x)=\hat a^{kl}(y^1)D_{y^k}\phi^i(y)D_{y^l}\phi^j(y)
D_{ij}u(x),
$$
where $y=\psi(x)$ and $\phi=\psi^{-1}$. Then there exist constants $N = N(d, \delta)$ and $\nu = \nu(d, \delta,p) \ge 1$ such that, for any $\kappa\ge 4$, $r > 0$ and $u \in  C^\infty_0$ we have
$$
\sum_{ij>1}\left(|u_{ij}-(u_{ij})_{Q_r}|^p\right)_{Q_r}
+\left(|u_t-(u_t)_{Q_r}|^p\right)_{Q_r}
$$
\begin{equation*}
\le N\kappa^{d+2}\left(|\hat P u|^p+|Du|^p\right)_{Q_{\nu\kappa r}}+N\kappa^{-\alpha p}
\left(|u_t|^p+|D^2 u|^p\right)_{Q_{\nu\kappa r}},
\end{equation*}
where
\begin{equation}
                            \label{eq9.23}
u_{ij}(x) = (D_{y^iy^j}v) (\psi(x)),\quad v(y) = u(\phi(y)).
\end{equation}
\end{lemma}

The next result can be considered as a generalization of Lemma \ref{lem4.1}.
\begin{theorem}
                                            \label{thm4.2}
Let $p\in (1,\infty)$, $\gamma > 0$, $\tau,\sigma \in (1,\infty)$ satisfying
$1/\tau+1/\sigma=1$. Let
$\alpha$ be the constant in
Lemma \ref{lemma 9.10.1} and $\nu=\nu(d,\delta)>1$ be the constant in Lemma \ref{lem4.1}. Assume $b^i=c=0$ and $u\in C_0^\infty$. Then under
Assumption \ref{assump2} ($\gamma$), for any
$r\in (0,\infty)$, $(t_0,x_0)\in \bR^{d+1}$ there exist a diffeomorphism $\psi\in\Psi$ and a positive constant $N=N(d,p,\delta,\tau)$, such that for any $\kappa\ge 4$,
$$
\sum_{ij>1}\left(|u_{ij}-(u_{ij})_{Q_r(t_0,x_0)}|^p\right)_{Q_r(t_0,x_0)}
+\left(|u_t-(u_t)_{Q_r(t_0,x_0)}|^p\right)_{Q_r(t_0,x_0)}
$$
$$
\le N\kappa^{d+2}
\left(|Pu|^p+|Du|^p\right)_{Q_{\nu\kappa r}(t_0,x_0)}
+N\kappa^{d+2}\gamma^{1/\sigma}
\left(|D^2 u|^{p\tau}\right)_{Q_{\nu\kappa r}(t_0,x_0)}^{1/\tau}
$$
\begin{equation}
                                \label{eq13.5.05}
+N\kappa^{-p\alpha}
\left(|u_t|^p\right)_{Q_{\nu\kappa r}(t_0,x_0)}+N(\kappa^{d+2}R^p+\kappa^{-p\alpha})
\left(|D^2u|^p\right)_{Q_{\nu\kappa r}(t_0,x_0)},
\end{equation}
provided that $u$ vanishes outside  $Q_{R}$ for some $R\in (0,R_0]$. Here $u_{ij}$ are defined in \eqref{eq9.23}.
\end{theorem}
\begin{proof}
We fix  $\kappa\ge 4$, and $r\in
(0,\infty)$.
  Choose $Q$ to be  $Q_{\kappa
r}(t_0,x_0)$ if $\nu\kappa r< R$ and   $Q_{ R }$ if $\nu\kappa
r\ge  R_0 $. Let $(t^*,x^*)$ be the center of $Q$. By Assumption \ref{assump2} ($\gamma$), we can find $\psi\in \Psi$ and $\bar a=\bar a(s)\in \mathcal A$ satisfying \eqref{eq3.07}. We set
$$
\hat a^{ij}(s)=\bar a^{kl}(s)(D_k\psi^i)(x^*)
(D_l\psi^j)(x^*),\quad y^*=\psi(x^*).
$$
By Lemma \ref{lem4.1} with a shift of the coordinates, for $ij>1$
$$
\left(|u_{ij}-(u_{ij})_{Q_r(t_0,x_0)}|^p\right)_{Q_r(t_0,x_0)}
+\left(|u_t-(u_t)_{Q_r(t_0,x_0)}|^p\right)_{Q_r(t_0,x_0)}
$$
\begin{equation}							
                                \label{13.5.42}
\le N\kappa^{d+2}\left(|\hat P u|^p+|Du|^p\right)_{Q_{\nu\kappa r}(t_0,x_0)}+N\kappa^{-\alpha p}
\left(|u_t|^p+|D^2 u|^p\right)_{Q_{\nu\kappa r}(t_0,x_0)},
\end{equation}
where $N$ depends only on $d$ and $\delta$.
By the definition of $\hat P$,
\begin{equation}							\label{13.5.52}
\int_{Q_{\nu\kappa r}(t_0,x_0)} |\hat Pu|^p \, dx\,dt
\le N \int_{Q_{\nu\kappa r}(t_0,x_0)} |Pu|^p \, dx\,dt
+N I+NJ,
\end{equation}
where
$$
I :=
\int_{Q_{\nu\kappa r}(t_0,x_0)\cap Q_R}
\big| \big( (D_{y^k}\phi^i D_{y^l}\phi^j)(\psi)
- (D_{y^k}\phi^i D_{y^l}\phi^j)(y^*)\big)\hat a^{kl}(\psi^1) D_{ij}u \big|^p \, dx\,dt
$$
$$
\le N\|(D_{y^k}\phi^i D_{y^l}\phi^j)(\psi)
- (D_{y^k}\phi^i D_{y^l}\phi^j)(y^*)\|^p_{L_\infty(Q_R)}
\int_{Q_{\nu\kappa r}(t_0,x_0)}|D_{ij}u|^p \, dx\,dt
$$
\begin{equation}
                            \label{eq10.38}
\le NR^p\int_{Q_{\nu\kappa r}(t_0,x_0)}|D_{ij}u|^p \, dx\,dt,
\end{equation}
and
$$
J:=\int_{Q_{\nu\kappa r}(t_0,x_0)\cap Q_R}
\big| (\bar a^{ij}(\psi^1)-a^{ij}(t,x)) D_{ij}u \big|^p \, dx\,dt.
$$
We used \eqref{eq10.54} in \eqref{eq10.38}.
By H\"older's inequality, we estimate $J$ by
\begin{equation}							\label{13.5.53}
J \le N J_1^{1/\sigma} J_2^{1/\tau},
\end{equation}
where
$$
J_1 = \sum_{i,j}
\int_{Q_{\nu\kappa r}(t_0,x_0) \cap Q_{ R}}
| \bar a^{ij}(\psi^1) - a^{ij} |^{p\sigma} \, dx\,dt,
\quad
J_2 = \int_{Q_{\nu\kappa r}(t_0,x_0)} |D^2 u|^{p\tau} \, dx\,dt.
$$
Due to Assumption \ref{assump2} ($\gamma$),
$$
J_1 \leq \sum_{i,j}\int_{Q }
| \bar a^{ij}(\psi^1) - a^{ij} |^{p\sigma} \, dx\,dt\leq N\gamma|Q|
\le  N  (\nu\kappa r)^{d+2} \gamma.
$$
This together with  \eqref{13.5.42}-\eqref{13.5.53} yields
\eqref{eq13.5.05}. The theorem is proved.
\end{proof}

From Theorem \ref{thm4.2} we obtain the following lemma in the same way as Lemma 3.4 \cite{Kr08} is deduced from Lemma 3.3 \cite{Kr08}.
\begin{lemma}
                                \label{lem4.3}
Let $q\in (1,\infty)$, $\gamma > 0$, $\kappa\ge 4$, $\tau,\sigma \in (1,\infty)$ satisfying $1/\tau+1/\sigma=1$. Let $\alpha$ be the constant in
Lemma \ref{lemma 9.10.1}. Suppose Assumption \ref{assump2} ($\gamma$) is satisfied. Assume $b^i=c=0$.
Then for any $n \in \bZ$ and $C \in \bC_n$ there exist a diffeomorphism $\psi \in\Psi$	 and
a constant $N = N(d, \delta, q,\tau)$ such that, for any $u\in C^\infty_0$ vanishing
outside $Q_R$ for some $R\in(0,R_0]$,
we have
\begin{equation}
					\label{eq3.10.35}
(|u_{t}-(u_{t})_C|)_C+\sum_{ij>1}(|u_{ij}-(u_{ij})_C|)_C\le N(\sfg)_C
\end{equation}
where $u_{ij}(t,x)$ are defined by \eqref{eq9.23} and
$$
\sfg=\kappa^{\frac {d+2} q}(\bM(|Pu|^q))^{\frac 1 q}
+\kappa^{\frac {d+2} q}(\bM(|Du|^q))^{\frac 1 q}
+\kappa^{\frac {d+2} q}\gamma^{\frac 1 {\sigma q}}(\bM(|D^2u|^{q\tau}))^{\frac 1 {q\tau}}
$$
\begin{equation*}
+\kappa^{-\alpha}(\bM(|u_t|^q))^{\frac 1 q}+(\kappa^{\frac {d+2} q}R+\kappa^{-\alpha})(\bM(|D^2u|^q))^{\frac 1 q}.
\end{equation*}
Moreover, we have
\begin{equation}
                            \label{eq11.22}
|u_t|+|D^2 u|\le N\sum_{ij>1}|u_{ij}|+N|u_t|+N|Du|+N|Pu|.
\end{equation}
\end{lemma}

\begin{theorem}
 					\label{thm4.4}
Let $p\in (1,\infty)$. Assume $b^i=c=0$. Then there exist positive constants
$\gamma$, $N$ and $R\in (0,1]$ depending only on $d,p$ and $\delta$ such that under Assumption \ref{assump2} ($\gamma$), for any $u\in C_0^\infty$
vanishing outside $Q_{RR_0}$, we have
\begin{equation}
 				\label{eq3.11.07}
\|u_t\|_{L_p}+\|D^2u\|_{L_p}\le N\|Pu\|_{L_p}+N\|Du\|_{L_p}.
\end{equation}
\end{theorem}
\begin{proof}
Set $f=Pu\in C_0^\infty$. Let $\gamma>0$, $\kappa\ge 4$ and $R\in (0,R_0]$ be constants to be specified later. Let $q=(1+p)/2\in (1,p)$ and $\tau=2(1+2p)/(3+3p)>1$ such that $p>q\tau$.
We take an $n\in \bZ$, a $C \in \bC_n$ and let $\psi  \in \Psi$ be the diffeomorphism from Lemma \ref{lem4.3}. Recall the definition of $u_{ij}$ in \eqref{eq9.23}. We then set
$$
U=|u_t|+|D^2 u|,\quad U^C=|u_t|+\sum_{ij>1}|u_{ij}|+|Du|+|f|,
\quad V=|u_t|+|D^2u|+|Du|+|f|.
$$
From \eqref{eq11.22}, we have $U\le NU^C$. By using the triangle inequality and \eqref{eq3.10.35},
$$
(|U^C-(U^C)_C|)_C\le 2(|u_t-(u_t)_C|)_C
+2\sum_{ij>1}(|u_{ij}-(u_{ij}|)_C)
$$
$$
+2(|Du-(Du)_C|)_C+2(|f-(f)_C|)_C
$$
$$
\le N(\sfg+|Du|+|f|)_C.
$$
Now by Theorem \ref{generalStein} with $H=\sfg+|Du|+f$, we get
$$
\|u_t\|_{L_p}^p+\|D^2 u\|_{L_p}^p\le N\|U\|_{L_p}^p\le N\|H\|_{L_p}\|V\|_{L_p}^{p-1}
\le N(\epsilon)\|H\|_{L_p}^p+\epsilon \|V\|_{L_p}^{p}.
$$
By taking a small $\epsilon>0$, it holds that
\begin{equation}
 					\label{eq3.11.00}
\|u_t\|_{L_p}^p+\|D^2 u\|_{L_p}^p
\le N\|\sfg\|_{L_p}^p+N\|Du\|_{L_p}^{p}+N\|f\|_{L_p}^{p}.
\end{equation}
We use the definition of $\sfg$ and the Hardy-Littlewood maximal function theorem (recall $p>q\tau>q$) to deduce from \eqref{eq3.11.00}
$$
\|u_t\|_{L_p}+\|D^2 u\|_{L_p}\le
N\kappa^{\frac {d+2} q}\|Pu\|_{L_p}
+N\kappa^{\frac {d+2} q}\|Du\|_{L_p}
+N\kappa^{-\alpha}\|u_t\|_{L_p}
$$
\begin{equation}
 					\label{eq3.11.05}
+N(\kappa^{\frac {d+2} q}\gamma^{\frac 1 {\sigma q}}+\kappa^{\frac {d+2} q}RR_0+\kappa^{-\alpha})\|D^2u\|_{L_p}.
\end{equation}
By choosing $\kappa$ sufficiently large, then $\gamma$ and $R$ sufficiently small in \eqref{eq3.11.05} such that
$$
N(\kappa^{\frac {d+2} q}\gamma^{\frac 1 {\sigma q}}+\kappa^{\frac {d+2} q}R+\kappa^{-\alpha})\le 1/2,
$$
we come to \eqref{eq3.11.07}. The theorem is proved.
\end{proof}

\begin{proof}[Proof of Theorem \ref{thm1}]
For $T=\infty$, the theorem follows from Theorem \ref{thm4.4} by using a partition of unity and an idea by S. Agmon; see, for instance, the proof of Theorem 1.4 \cite{Kr08}. For general $T\in (-\infty,+\infty]$, we again use the argument at the end of the proof of Theorem \ref{thm2.2}.
\end{proof}

\mysection{Sobolev spaces with mixed norms}
		\label{sec5}
In this section we consider parabolic equations in Sobolev spaces with mixed norms in the spirit of \cite{Krylov_2007_mixed_VMO}. As pointed out in \cite{Krylov_2007_mixed_VMO}, the interest in results concerning equations in spaces with mixed Sobolev norms arises, for example, when one wants to get better regularity of traces of solutions for each time slide (see, for instance, \cite{SoftWeid06, Weid02} and references therein).

Our objective is to prove the following theorem, which generalizes Theorem \ref{thm1}.

\begin{theorem}
                            \label{thm3}
For any
$1<p\le q<\infty$ there exists a $\gamma =\gamma(d,\delta,p,q) >0$
such that under Assumption
\ref{assump2} ($\gamma$) for any $T\in (-\infty,+\infty]$ the following holds.

i) For any $u\in W^{1,2}_{q,p}(\bR^{d+1}_T)$,
$$
\lambda\|u\|_{L_{q,p}(\bR^{d+1}_T)}+\sqrt{\lambda}
\|Du\|_{L_{q,p}(\bR^{d+1}_T)}+\|D^{2}u\|_{L_{q,p}(\bR^{d+1}_T)}
+\|u_t\|_{L_{q,p}(\bR^{d+1}_T)}
$$
\begin{equation*}
\leq N\|Pu-\lambda u\|_{L_{q,p}(\bR^{d+1}_T)},
\end{equation*}
provided that $\lambda\geq \lambda_0$, where $\lambda_0
\geq0$ and $N$ depend only
 on $d,\delta,p,q,K$, and $R_0$.

ii) For any $\lambda> \lambda_0$ and $f\in L_{q,p}(\bR^{d+1}_T)$, there exists a unique solution
$u\in W^{1,2}_{q,p}(\bR^{d+1}_T)$ of equation \eqref{parabolic} in $\bR^{d+1}_T$.
\end{theorem}

Since the case $p=q$ has been covered in Theorem \ref{thm1}, in the sequel we assume $p<q$. We make a few preparations before the proof.

\begin{lemma}
                                        \label{lem5.3}
Let $p,q\in(1,\infty)$, $b=c=0$. Then there exists a constant $\gamma_0=\gamma_0(d,p,q,\delta)>0$ such that under Assumption \ref{assump2} ($\gamma_0$),
for any $r\in (0,R_0]$ and $u\in W^{1,2}_{q,\text{loc}}$ satisfying $P u=0$
in $Q_{2r}$ we have $D^2 u \in L_p(Q_r)$ and
$$
(|D^2 u|^p)_{Q_r}^{1/p}\leq N (|D^2 u|^q)_{Q_{2r}}^{1/q},
$$
where $N$ depends only on $d,p,q$ and $\delta$.
\end{lemma}
\begin{proof}
First we assume $R_0=1$. In this case, the lemma is proved in Corollary 6.4 of \cite{Krylov_2007_mixed_VMO} with the only difference that the coefficients $a^{ij}$ are assumed to be in ${\rm VMO}_x$ in that paper. The proof of Corollary 6.4 \cite{Krylov_2007_mixed_VMO} uses the $L_p$ solvability of equations with ${\rm VMO}_x$ coefficients. Since the solvability is already established with coefficients satisfying Assumption \ref{assump2} ($\gamma_0$) with a $\gamma_0$ depending on $d,p,q$ and $\delta$, we can just reproduce the proof with almost no change.

For general $R_0\in (0,1]$, we make a change of variables $(t,x)\to (R_0^2t, R_0 x)$ and notice that the new coefficients satisfy Assumption \ref{assump2} ($\gamma_0$) with $R_0$ replaced by $1$. The lemma is proved.
\end{proof}

The next theorem improves Theorem \ref{thm4.2}.

\begin{theorem}
                                            \label{thm5.4}
Let $p\in (1,\infty)$ and $\gamma > 0$. Let
$\alpha$ be the constant in
Lemma \ref{lemma 9.10.1}, $\nu=\nu(d,\delta)>1$ be the constant in Lemma \ref{lem4.1} and $\gamma_0=\gamma_0(d,p,2p,\delta)$ be the constants in Lemma \ref{lem5.3}. Assume $b^i=c=0$ and $u\in C_0^\infty$. Then under Assumption \ref{assump2} ($\gamma$) with $\gamma\in (0,\gamma_0]$ the following is true. For any $R\in (0,R_0]$, $\kappa\ge 8$, $r\in (0,R\kappa^{-1}\nu^{-1}]$ and $(t_0,x_0)\in \bR^{d+1}$, there exists a diffeomorphism $\psi\in\Psi$, such that
$$
\sum_{ij>1}\left(|u_{ij}-(u_{ij})_{Q_r(t_0,x_0)}|^p\right)_{Q_r(t_0,x_0)}
+\left(|u_t-(u_t)_{Q_r(t_0,x_0)}|^p\right)_{Q_r(t_0,x_0)}
$$
$$
\le N\kappa^{d+2}
\left(|Pu|^p+|Du|^p\right)_{Q_{\nu\kappa r}(t_0,x_0)}
+N\kappa^{-p\alpha}
\left(|u_t|^p\right)_{Q_{\nu\kappa r}(t_0,x_0)}
$$
\begin{equation}
                                \label{eq06.12.01}
+N(\kappa^{d+2}\gamma^{1/2}+\kappa^{d+2}R^p+\kappa^{-p\alpha})
\left(|D^2u|^p\right)_{Q_{\nu\kappa r}(t_0,x_0)},
\end{equation}
where $u_{ij}$ are defined in \eqref{eq9.23} and $N=N(d,p,\delta)>0$.
\end{theorem}
\begin{proof}
We assume without loss of generality that $(t_0,x_0)=(0,0)$. We may also assume that $a^{ij}$ are infinitely differentiable
by using standard mollifications if necessary.

Let $f:=Pu$. Take a cutoff function $\eta\in C_0^\infty$ such that $\eta=1$ on $Q_{\nu\kappa r/2}$ and $\eta=0$ outside the closure of $Q_{\nu\kappa r}\cup (-Q_{\nu\kappa r})$. Due to Theorem \ref{thm1b} there exists a unique solution $w\in W^{1,2}_p((-1,0)\times \bR^d)$ of
$$
Pw=f\eta\,\,\,\text{on}\,\,(-1,0)\times \bR^d,\quad w(-1,\cdot)=0.
$$
By the classical theory, we have $w\in C^\infty((-1,0)\times \bR^d)$. Let $h=u-w$, which is also smooth and satisfies
$$
Ph=f(1-\eta)\,\,\,\text{on}\,\,(-1,0)\times \bR^d,
\quad Ph=0\,\,\,\text{on}\,\,Q_{\nu\kappa r/2}.
$$

First we estimate $h$. Choose $Q$ to be  $Q_{\kappa
r/2}$. By Assumption \ref{assump2} ($\gamma$), we can find $\psi\in \Psi$ and $\bar a=\bar a(s)\in \mathcal A$ satisfying \eqref{eq3.07}. By repeating the proof of Theorem \ref{thm4.2} with $h$ in place of $u$ (recall $\kappa/2\ge 4$), we get
$$
\sum_{ij>1}\left(|h_{ij}-(h_{ij})_{Q_r}|^p\right)_{Q_r}
+\left(|h_t-(h_t)_{Q_r}|^p\right)_{Q_r}
$$
$$
\le N\kappa^{d+2}
\left(|Dh|^p\right)_{Q_{\nu\kappa r/2}}
+N\kappa^{d+2}\gamma^{1/2}
\left(|D^2 h|^{2p}\right)_{Q_{\nu\kappa r/2}}^{1/2}
$$
\begin{equation*}
+N\kappa^{-p\alpha}
\left(|h_t|^p\right)_{Q_{\nu\kappa r/2}}+N(\kappa^{d+2}R^p+\kappa^{-p\alpha})
\left(|D^2h|^p\right)_{Q_{\nu\kappa r/2}}
\end{equation*}
$$
\le N\kappa^{d+2}
\left(|Dh|^p\right)_{Q_{\nu\kappa r}}
+N\kappa^{d+2}\gamma^{1/2}
\left(|D^2 h|^{p}\right)_{Q_{\nu\kappa r}}
$$
\begin{equation}
                                \label{eq06.5.05}
+N\kappa^{-p\alpha}
\left(|h_t|^p\right)_{Q_{\nu\kappa r}}+N(\kappa^{d+2}R^p+\kappa^{-p\alpha})
\left(|D^2h|^p\right)_{Q_{\nu\kappa r}},
\end{equation}
where $h_{ij}$ are defined in the same way as $u_{ij}$. In the last inequality we used $Ph=0$ on $Q_{\nu\kappa r/2}$ and Lemma \ref{lem5.3}.

Next we estimate $w$. Due to Theorem \ref{thm1b}, we have
$$
\|w\|_{W^{1,2}_p((-1,0)\times \bR^d)}\le N\|f\eta\|_{L_p((-1,0)\times \bR^d)}\le N\|f\|_{L_p(Q_{\nu\kappa r})}.
$$
Therefore,
\begin{equation}
                                        \label{eq4.27b}
(|w_t|^p+|Dw|^p+|D^2 w|^p)_{Q_r}\le N\kappa^{d+2}(|f|^p)_{Q_{\nu\kappa r}},
\end{equation}
\begin{equation}
                                        \label{eq4.28b}
(|w_t|^p+|Dw|^p+|D^2 w|^p)_{Q_{\nu\kappa r}}\le N(|f|^p)_{Q_{\nu\kappa r}}.
\end{equation}

Since $|w_{ij}|\le N|Dw|+N|D^2 w|$, combining \eqref{eq06.5.05}, \eqref{eq4.27b} and \eqref{eq4.28b} we then get by using the triangle inequality,
$$
\sum_{ij>1}\left(|u_{ij}-(u_{ij})_{Q_r}|^p\right)_{Q_r}
+\left(|u_t-(u_t)_{Q_r}|^p\right)_{Q_r}
$$
$$
\le N\sum_{ij>1}\left(|h_{ij}-(h_{ij})_{Q_r}|^p\right)_{Q_r}
+N\left(|h_t-(h_t)_{Q_r}|^p\right)_{Q_r}+N(|w_t|^p+|Dw|^p+|D^2w|^p)_{Q_r}
$$
$$
\le N\kappa^{d+2}
\left(|Dh|^p\right)_{Q_{\nu\kappa r}}
+N(\kappa^{d+2}\gamma^{1/2}+\kappa^{d+2}R^p+\kappa^{-p\alpha})
\left(|D^2 h|^{p}\right)_{Q_{\nu\kappa r}}
$$
$$
+N\kappa^{-p\alpha}
\left(|h_t|^p\right)_{Q_{\nu\kappa r}}+N\kappa^{d+2}(|f|^p)_{Q_{\nu\kappa r}}
$$
$$
\le N\kappa^{d+2}
\left(|Du|^p\right)_{Q_{\nu\kappa r}}
+N(\kappa^{d+2}\gamma^{1/2}+\kappa^{d+2}R^p+\kappa^{-p\alpha})
\left(|D^2 u|^{p}\right)_{Q_{\nu\kappa r}}
$$
$$
+N\kappa^{-p\alpha}
\left(|u_t|^p\right)_{Q_{\nu\kappa r}}+N\kappa^{d+2}(|f|^p)_{Q_{\nu\kappa r}},
$$
which is exactly the right-hand side of \eqref{eq06.12.01}. The theorem is proved.
\end{proof}

The next corollary can be deduced from Theorem \ref{thm5.4}. 
\begin{corollary}
                                                \label{cor5.5}
Under the assumptions of Theorem \ref{thm5.4}, for any $R\in (0,R_0]$, $\kappa\ge 8$ and any interval $[S,T)$ such that $(T-S)^{1/2}=:r\in (0,R\kappa^{-1}\nu^{-1}]$, we have
$$
\dashint_{(S,T)}\dashint_{(S,T)}\abs{
\varphi(t)-\varphi(s)}^p\,dt\,ds
\le
N\kappa^{d+2}\dashint_{(T-(\nu\kappa r)^2,T)} \zeta(t)^p\,dt
$$
\begin{equation}
                                                    \label{eq12.47pm}
+
N(\kappa^{d+2}R^p+\kappa^{-p\alpha}+
\kappa^{d+2}\gamma^{1/2})
\dashint_{(T-(\nu\kappa r)^2,T)}\rho(t)^p\,dt,
\end{equation}
where $N=N(d,p,\delta)>0$, and
$$
\zeta(t)^p=\int_{\bR^d}|Pu(t,x)|^p+|Du(t,x)|^p\,dx,
\quad \rho(t)^p=\int_{\bR^d}|D^2 u(t,\cdot)|^p+|u_t(t,\cdot)|^p\,dx,
$$
$$
\varphi(t)^p=\int_{\bR^d}\dashint_{B_r(y)}|u_t(t,\cdot)|^p+\sum_{ij>1}|u^{(y)}_{ij}(t,\cdot)|^p\,dx\,dy.
$$
Here for each $y\in \bR^d$, $u^{(y)}_{ij}$ are defined in \eqref{eq9.23} with $Q=Q_{\kappa r/2}(T,y)$ in Assumption \ref{assump2}.
\end{corollary}
\begin{proof}
By the triangle inequality, the left-hand side of \eqref{eq12.47pm} is less than a constant $N$ times
$$
\dashint_{(S,T)}\dashint_{(S,T)}\int_{\bR^d}
\dashint_{B_r(y)}|u_t(t,\cdot)-u_t(s,\cdot)|^p+\sum_{ij>1}|u^{(y)}_{ij}(t,\cdot)-u^{(y)}_{ij}(s,\cdot)|^p\,dx\,dy\,dt\,ds
$$
$$
\le N\int_{\bR^d}
\sum_{ij>1}\left(|u^{(y)}_{ij}-(u^{(y)}_{ij})_{Q_r(T,y)}|^p
\right)_{Q_r(T,y)}
+\left(|u_t-(u_t)_{Q_r(T,y)}|^p\right)_{Q_r(T,y)}\,dy.
$$
Due to Theorem \ref{thm5.4}, the last expression is less than $N$ times
$$
\int_{\bR^d}\kappa^{d+2}
\left(|Pu|^p+|Du|^p\right)_{Q_{\nu\kappa r}(T,y)}
$$
$$
+(\kappa^{d+2}(R^p+\gamma^{\frac{1}{2}})+\kappa^{-p\alpha})
\left(|D^2 u|^{p}+|u_t|^{p}\right)_{Q_{\nu\kappa r}(T,y)}\,dy,
$$
which is less than the right-hand side of \eqref{eq12.47pm}.
\end{proof}



Finally, we prove the following estimate which implies Theorem \ref{thm3} in the same way as Theorem \ref{thm4.4} implies Theorem \ref{thm1}.

\begin{theorem}
 				\label{thm5.7}
Let $1<p<q<\infty$. Assume $b^i=c=0$. Then there exist positive constants
$\gamma$, $N$ and $R\in (0,1]$ depending only on $d,p,q$ and $\delta$ such that under Assumption \ref{assump2} ($\gamma$), for any $u\in C_0^\infty$
vanishing outside $(-R^4R_0^2,0)\times \bR^d$, we have
\begin{equation}
 				\label{eq6.5.56}
\|u_t\|_{L_{q,p}}+\|D^2u\|_{L_{q,p}}\le N\|Pu\|_{L_{q,p}}+N\|Du\|_{L_{q,p}}.
\end{equation}
\end{theorem}
\begin{proof}
Set $f=Pu$. Recall the definitions of $\varphi$, $\rho$ and $\zeta$ in Corollary \ref{cor5.5}. Let $\gamma\in (0,\gamma_0]$, $\kappa\in [8,\infty)$ and $R\in (0,1]$ be numbers to be chosen later, where $\gamma_0=\gamma_0(d,p,2p,\delta)$ is taken from Lemma \ref{lem5.3}. Assume $u$ vanishes outside $(-R^4R_0^2,0)\times \bR^d$. Denote $f=Pu$ and
$$
\sfg=\sfg(t)=\left(\kappa^{\frac {d+2} p}R_0R+\kappa^{-\alpha}+
\kappa^{\frac {d+2} p}\gamma^{\frac 1 {2p}}+(R\kappa)^{2(1-\frac 1 p)}\right)(\bM (\rho^p))^{\frac 1 p}
$$
$$
+N\left(\kappa^{\frac {d+2} p}+(R\kappa)^{2(1-\frac 1 p)}\right)(\bM (\zeta^p))^{\frac 1 p}.
$$

Let $\bC_n$ be the filtration of partitions given by dyadic intervals
$$
\left\{[j2^{-2n},(j+1)2^{-2n}),j\in\bZ\right\}.
$$
For any $C=[S,T)\in \bC_n$, we set
$$
U(t)=\|u_t(t,\cdot)\|_{L_p(\bR^d)}+\|D^2 u(t,\cdot)\|_{L_p(\bR^d)},
$$
$$
V(t)=\|u_t(t,\cdot)\|_{L_p(\bR^d)}+\|D^2 u(t,\cdot)\|_{L_p(\bR^d)}+\|Du(t,\cdot)\|_{L_p(\bR^d)}
+\|f(t,\cdot)\|_{L_p(\bR^d)},
$$
and
\begin{equation*}	
U^C(t)=\left\{
  \begin{array}{ll}
    \varphi(t)+\|Du(t,\cdot)\|_{L_p(\bR^d)}
+\|f(t,\cdot)\|_{L_p(\bR^d)} & \hbox{if $2^{-n}\le RR_0\kappa^{-1}\nu^{-1}$} \\
    V(t) & \hbox{otherwise}
  \end{array}.
\right.
\end{equation*}
It is easy to see that $U\le NU^C\le NV$ in $C$. We claim
\begin{equation}
                        \label{eq2.41pm}
(U^C-(U^C)_C)_C\le N\left(\sfg+\|Du(t,\cdot)\|_{L_p(\bR^d)}
+\|f(t,\cdot)\|_{L_p(\bR^d)}\right)_C.
\end{equation}
Indeed, if $2^{-n}\le R_0R\kappa^{-1}\nu^{-1}$, by a shift of the origin we get \eqref{eq2.41pm} from Corollary \ref{cor5.5}. Otherwise, we have
$$
(U^C-(U^C)_{C})_{C}\le 2\dashint_{(S,T)}\chi_{(-R_0^2R^4,0)}|V(t)|\,dt
$$
$$
\le 2\left(\dashint_{(S,T)}\chi_{(-R_0^2R^4,0)}\,dt\right)^{1-1/p}
\left(\dashint_{(S,T)}|V(t)|^p\,dt\right)^{1/p}
$$
$$
\le
N(R\kappa)^{2(1-1/p)}\left(\bM (\rho^p+\zeta^p)(t_0)\right)^{1/p},
$$
for any $t_0\in C$. This proves the claim.

Now by Theorem \ref{generalStein} with $d=0$, $q$ in place of $p$ and
$$
H(t):=\sfg(t)+\|Du(t,\cdot)\|_{L_p(\bR^d)}
+\|f(t,\cdot)\|_{L_p(\bR^d)},
$$
we get
$$
\|u_t\|_{L_{q,p}}^q+\|D^2 u\|_{L_{q,p}}^q
\le N\|U\|_{L_{q}}^q\le N\|H\|_{L_q}\|V\|_{L_q}^{q-1}
\le N(\epsilon)\|H\|_{L_q}^q+\epsilon \|V\|_{L_q}^{q}.
$$
By taking a small $\epsilon>0$, it holds that
\begin{equation}
 					\label{eq6.11.00}
\|u_t\|_{L_{q,p}}^q+\|D^2 u\|_{L_{q,p}}^q
\le N\|\sfg\|_{L_q}^q+N\|Du\|_{L_{q,p}}^{q}+N\|f\|_{L_{q,p}}^{q}.
\end{equation}
We use the definition of $\sfg$ and the Hardy-Littlewood maximal function theorem (recall $p>q$) to deduce from \eqref{eq6.11.00}
$$
\|u_t\|_{L_{q,p}}+\|D^2 u\|_{L_{q,p}}\le
N\left(\kappa^{\frac {d+2} p}+(R\kappa)^{2(1-\frac 1 p)}\right)(\|f\|_{L_{q,p}}+\|Du\|_{L_{q,p}})
$$
\begin{equation}
 					\label{eq6.11.05}
+N\left(\kappa^{\frac {d+2} p}R_0R+\kappa^{-\alpha}+
\kappa^{\frac {d+2} p}\gamma^{\frac 1 {2p}}+(R\kappa)^{2(1-\frac 1 p)}\right)(\|u_t\|_{L_{q,p}}+\|D^2u\|_{L_{q,p}}).
\end{equation}
By choosing $\kappa$ sufficiently large, then $\gamma$ and $R$ sufficiently small in \eqref{eq6.11.05} such that
$$
N\left(\kappa^{\frac {d+2} p}R+\kappa^{-\alpha}+
\kappa^{\frac {d+2} p}\gamma^{\frac 1 {2p}}+(R\kappa)^{2(1-\frac 1 p)}\right)\le 1/2,
$$
we come to \eqref{eq6.5.56}. The theorem is proved.
\end{proof}

\section*{Acknowledgement}

The author is sincerely grateful to Nicolai V. Krylov for very helpful
comments on the first draft of the article.


\end{document}